\documentclass[11pt, a4paper, english]{amsart}

\usepackage{amsmath,amssymb,amsthm,comment}
\usepackage{bbm}
\usepackage[cal=euler,scr=rsfs]{mathalfa}
\usepackage[all]{xy}
\usepackage[colorlinks=true,backref=page]{hyperref}
\usepackage{cleveref}

\numberwithin{equation}{section}

\theoremstyle{theorem}
\newtheorem{theorem}{Theorem}[section]
\newtheorem{proposition}[theorem]{Proposition}
\newtheorem{lemma}[theorem]{Lemma}
\newtheorem{corollary}[theorem]{Corollary}
\newtheorem{conjecture}[theorem]{Conjecture}
\newtheorem{problem}[theorem]{Problem}

\theoremstyle{definition}

\theoremstyle{remark}

\newcommand{\Z}{\mathbb{Z}}

\newcommand{\R}{\mathbb{R}}
\newcommand{\C}{\mathbb{C}}

\newcommand{\interior}{\operatorname{int}}

\SelectTips{cm}{}

\title{On the embeddability of skeleta of manifold triangulations}

\author{Daisuke Kishimoto}
\address{Faculty of Mathematics, Kyushu University, Fukuoka 819-0395, Japan}
\email{kishimoto@math.kyushu-u.ac.jp}

\author{Takahiro Matsushita}
\address{Department of Mathematical Sciences, Shinshu University, Matsumoto, Nagano 390-8621, Japan}
\email{matsushita@shinshu-u.ac.jp}

\date{\today}



\begin{document}

\maketitle

\begin{abstract}
We show a criterion for a skeleton of a manifold triangulation being embeddable into Euclidean space in terms of the complement of a submanifold. As an application, we obtain embeddability of a $(q-1)$-skeleton of a triangulation of an $S^p$-bundle over $S^q$ into $\R^{p+q}$.
\end{abstract}

%%%%% Section 1 %%%%%

\section{Introduction}\label{Introduction}

Throughout the paper, manifolds and submanifolds will be smooth and boundaryless, and an embedding will mean a homeomorphism into unless otherwise specified.

Recall that the van Kampen-Flores theorem \cite{F, vK} states that the $d$-skeleton of a $(2d+2)$-simplex is not embeddable into $\R^{2d}$, where every $d$-dimensional simplicial complex is embeddable into $\R^{2d+1}$. One can interpret this as the nonembeddability of the $d$-skeleton of a very special triangulation of a $(2d+1)$-sphere, and Nevo and Wagner \cite{NW} later proved that the $d$-skeleton of any triangulation of a $(2d+1)$-sphere is nonembeddable into $\R^{2d}$. Then one may naively ask:

\begin{problem}
  Which connected closed manifold triangulations have the $d$-skeleta embeddable into $\R^{2d}$?
\end{problem}

Let $M$ be a manifold. According to the van Kampen-Flores theorem, if the dimension of $M$ is at least $2d+2$, then the $d$-skeleton of every triangulation of $M$ is nonembeddable into $\mathbb{R}^{2d}$. In \cite{KM1, KM2}, the authors showed that when the dimension of $M$ is $2d+1$, if $M$ is either a $\mathbb{Z}/2$-homology sphere (cf. \cite{HKTT}) or the total Stiefel-Whitney class $w(M)$ is nontrivial, then the $d$-skeleton of any triangulation of $M$ is nonembeddable into $\mathbb{R}^{2d}$. We also showed that the $d$-skeleton of any triangulation of a closed orientable $2d$-manifold $M$ with odd Euler characteristic is nonembeddable into $\R^{2d}$. When $d=1$, the Euler characteristic condition can be weakened because there is a folklore fact in topological graph theory that the $1$-skeleton of any triangulation of a closed surface of genus $\ge 1$ is nonembeddable into $\mathbb{R}^{2}$ (see Section 1 of \cite{KM2} for example).

%Note that the examples mentioned above are all cases where the manifold $M$ cannot be embedded into $\mathbb{R}^{2d}$.

Note that all the results above are on nonembeddability, and that in these examples the entire manifold $M$ is nonembeddable into $\R^{2d}$. The reason for this is obvious: if the closed manifold $M$ is embeddable into $\mathbb{R}^{2d}$, then the $d$-skeleton of any triangulation of $M$ is embeddable into $\mathbb{R}^{2d}$. Then the next problem is whether the converse holds true or not. Namely, if $M$ is nonembeddable into $\mathbb{R}^{2d}$, can we also say that the $d$-skeleton of any triangulation of $M$ is also nonembeddable into $\mathbb{R}^{2d}$? However, there is a trivial counterexample: the $2d$-dimensional sphere is nonembeddable into $\mathbb{R}^{2d}$, but the $d$-skeleton of any triangulation of it can obviously be embedded into $\mathbb{R}^{2d}$. In particular, when $\dim M=2$ and $d=1$, the embeddability of the $d$-skeleta of triangulations of closed surfaces are completely solved because the $1$-skeleton of any triangulation of a surface of genus $\ge 1$ is nonembeddable into $\R^2$, as mentioned above. Thus in this paper, we mainly consider this problem when $M$ is a closed manifold of dimension $2d$ for $d\ge 2$, in which case $M$ is nonembeddable into $\R^{2d}$. We will prove a criterion for embeddability of the $d$-skeleton of any smooth triangulation of $M$ into $\R^{2d}$ in terms of the complement of a submanifold of $M$. This will yield a variety of affirmative examples to the problem. In particular, in contrast to the $d = 1$ case, we show that there are infinitely many examples of closed $2d$-manifolds having triangulations whose $d$-skeleta are embeddable into $\R^{2d}$ (Corollary \ref{corollary infinite examples}).

\begin{comment}
The authors \cite{KM1,KM2} studied this problem, and in particular, obtained a criterion of nonembeddability of skeleta of manifold triangulations in terms of the Stiefel-Whitney classes. As an application, one gets that a $d$-skeleton of any manifold triangulation of a closed orientable $2d$-manifold with odd Euler characteristic is not embeddable into $\R^{2d}$. On the other hand, one can easily see an embeddable example of the $d$-skeleton of a $2d$-manifold triangulation. For example, the $d$-skeleton of any triangulation of $S^{2d}$ is embeddable into $\R^{2d}$ because the complement of a point in $S^{2d}$ is embeddable into $\R^{2d}$. In this paper, we generalize this simple observation and give a criterion of embeddability of a skeleton of a manifold triangulation in terms of embeddability of the complement of a submanifold.
\end{comment}

To state our main result precisely, we prepare some notion. We say that a simplicial complex $K$ is a \emph{smooth triangulation} of a manifold $M$ if there is a homeomorphism $K\to M$ which restricts to a smooth embedding on each simplex of $K$. The classical result of Whitehead \cite{W} shows that every manifold has a smooth triangulation, where the original Whitehead's result is for $C^1$-triangulations but the proof works verbatim for smooth triangulations.

\begin{theorem}
  \label{main}
  Let $M$ be an $n$-manifold. If there is a $p$-dimensional compact submanifold $L$ of $M$ such that $M-L$ is embeddable into $\R^m$, then for any smooth triangulation $K$ of $M$, the $(n-p-1)$-skeleton $K_{n-p-1}$ is embeddable into $\R^m$.
\end{theorem}

We show an application example of Theorem \ref{main}.

\begin{corollary}
  \label{sphere bundle}
  Let $M$ be an $S^p$-bundle over $S^q$. The $(q-1)$-skeleton of any smooth triangulation of $M$ is embeddable into $\R^{p+q}$.
\end{corollary}

\begin{proof}
  Let $\pi\colon M\to S^q$ denote the projection, and take any point $x\in S^q$. Then $L=\pi^{-1}(x)\cong S^p$ is a submanifold of $M$ such that $M-L\cong\R^q\times S^p$. Now there is an embedding
  \[
    \R^q\times S^p\to\R^{p+q},\quad((x_1,\ldots,x_q),y)\mapsto(e^{x_1}y,x_2,\ldots,x_q).
  \]
  Then by Theorem \ref{main}, the statement follows.
\end{proof}

Note that if $p+q=2d$ and $p<q$, then the $d$-skeleton of any smooth triangulation of an $S^p$-bundle over $S^q$ is embeddable into $\R^{2d}$ by Corollary \ref{sphere bundle}.
Recall that the $(2i)$-th Stiefel-Whitney class of $\C P^d$ is given by
\[
  w_{2i}(\C P^d)=\binom{d+1}{i}x^i%(1+x)^{d+1}
\]
where $x$ is the generator of $H^2(\C P^d;\Z/2)\cong\Z/2$. Then $w_{2i}(\C P^d)\ne 0$ for some $i\ge 1$ unless $d=2^k-1$. Hence by \cite[Theorem 1.2]{KM2}, the $d$-skeleton of any triangulation of $\C P^d$ is nonembeddable into $\R^{2d}$ unless $d=2^k-1$ for some $k\ge 1$. On the other hand, by the above observation, the $1$-skeleton of any triangulation of $\C P^1=S^2$ is embeddable into $\R^2$, and by Corollary \ref{sphere bundle}, the $3$-skeleton of any smooth triangulation of $\C P^3$ is embeddable into $\R^6$ as it is an $S^2$-bundle over $S^4$. Then we pose:

\begin{conjecture}
  The $d$-skeleton of any triangulation of $\C P^d$ is embeddable into $\R^{2d}$ whenever $d=2^k-1$ for some $k\ge 1$.
\end{conjecture}

We will prove the following criterion, which can be used to generate an infinite number of examples to which Theorem~\ref{main} can be applied.

\begin{proposition}
  \label{connected sum}
  For $i=1,\ldots,k$, let $L_i$ be a $p$-dimensional submanifold of a connected $n$-manifold $M_i$ such that $M_i-L_i$ is embeddable into $\R^n$. Then there is a $p$-dimensional compact submanifold $L$ of the connected sum $M_1\#\cdots\# M_k$ such that $M_1\#\cdots\# M_k-L$ is embeddable into $\R^n$.
\end{proposition}

By Corollary \ref{sphere bundle}, if $n = 2d = p + q$ and $p<q$, then we may set $(M_i,L_i)=(S^p \times S^q, S^p \times *)$ for $i=1,2,\ldots,k$ in Proposition \ref{connected sum}, hence we conclude the following, where we must have $d\ge 2$:

\begin{corollary} \label{corollary infinite examples}
For $d \ge 2$, there are infinitely many closed $2d$-manifolds having triangulations whose $d$-skeleta are embeddable into $\R^{2d}$.
\end{corollary}

\subsection*{Acknowledgement}

The authors thank the anonymous referee for useful comments, which improved the readability of the manuscript. The first author was supported by JSPS KAKENHI Grant Number JP22K03284, and the second author was supported by JSPS KAKENHI Grant Number JP23K12975.

%%%%% Section 2 %%%%%

\section{Proof}

%\begin{proof}
%  [Proof of Theorem \ref{main 1}]
%  Let $M$ be an $S^p$-bundle over $S^q$, and let $\pi\colon M\to S^q$ denote the projection. Since $K$ is a smooth triangulation of $M$, there is a homeomorphism $h\colon K\to M$ which is smooth on each simplex of $K$, and so for each simplex $\sigma$ of $K$ with , $h(\mathrm{Int}(\sigma))$ is a submanifold of $M$. Then it follows from Sard's theorem that for $\dim\sigma<q$, $\pi(h(\mathrm{Int}(\sigma)))$ has measure zero, hence $\pi(h(K_{q-1}))$ has measure zero too as $K$ is a finite complex. Thus there is a point $x\in S^q$ such that $x\not\in\pi(h(K_{q-1}))$. Observe that $\pi^{-1}(S^q-x)\cong S^p\times(\R^q-0)$ and there is an embedding
%  \[
%    S^p\times(\R^q-0)\to\R^{p+q},\quad(x,(y_1,\ldots,y_q))\mapsto(y_1x,y_2,\ldots,y_q).
%  \]
%  Then since $K_{q-1}\subset\pi^{-1}(S^q-x)$, $K_{q-1}$ is embeddable into $\R^{p+q}$.
%\end{proof}

First, we recall the transversality theorem (Theorem~\ref{transversality theorem}).
%\cite[2.1. Transversality Theorem, p. 74]{H}.
Let $M$ and $N$ be manifolds. Let $C^\infty(M,N)$ denote the space of smooth maps from $M$ to $N$, equipped with the $C^\infty$-topology. Let $L$ be a submanifold of $N$. Recall that a smooth map $f\colon M\to N$ is \emph{transverse} to $L$ if for any $x\in M$ with $f(x)\in L$, one has
\[
  T_{f(x)}N=f_*(T_xM)+T_{f(x)}L.
\]
Observe that for $\dim M+\dim L<\dim N$, a smooth map $f\colon M\to N$ is transverse to $L$ if and only if $f(M)\cap L=\emptyset$. Recall that a subset of a topological space is called \emph{residual} if it contains the intersection of countably many dense open sets. It is well known that $C^\infty(M,N)$ is a Baire space, that is, the intersection of countably many residual sets of $C^\infty(M,N)$ is dense.

\begin{theorem}
  [{see \cite[Chapter~3, Theorem~2.1 (a), p.74]{H}}]
  \label{transversality theorem}
  Let $M$ and $N$ be manifolds, and let $L$ be a submanifold of $N$. Then the subset of $C^\infty(M,N)$ consisting of transverse maps to $L$ is residual.
\end{theorem}

Second, we recall the isotopy extension theorem (Theorem~\ref{isotopy extension theorem}). Let $M$ and $N$ be manifolds. An \emph{isotopy} is a smooth map $h_t\colon M\times[0,1]\to N$ such that $h_t$ is a smooth embedding for each $t\in[0,1]$.

\begin{theorem}
  [{see \cite[Chapter~8, Theorem~1.3, p.180]{H}}]
  \label{isotopy extension theorem}
  Let $M$ be a compact submanifold of a manifold $N$. Then any isotopy $h \colon M \times [0,1] \to N$, $(x,t) \mapsto h_t(x)$ with $h_0$ being the inclusion of $M$ into $N$ extends to a smooth map $H \colon N \times [0,1] \to N$, $(x,t) \mapsto H_t(x)$ such that $H_0 = {\rm id}_N$ and $H_t$ is a diffeomorphism of $N$ for each $t\in[0,1]$.
\end{theorem}

For manifolds $M$ and $N$, let $\mathrm{Emb}(M,N)$ denote the subspace of $C^\infty(M,N)$ consisting of smooth embeddings.

\begin{lemma}
  \label{Emb local}
  Let $M$ be a closed submanifold of a manifold $N$. Then $\mathrm{Emb}(M,N)$ is a locally path-connected open subset of $C^\infty(M,N)$.
\end{lemma}

\begin{proof}
  Clearly, $\mathrm{Emb}(M,N)$ is open in $C^\infty(M,N)$. By \cite[Theorem 44.1]{KrM}, $\mathrm{Emb}(M,N)$ is a Fr\'echet manifold locally modeled on the Fr\'echet space of smooth sections of $TN\vert_M$. Thus the statement follows.
\end{proof}

Now we are ready to prove Theorem \ref{main}.

\begin{proof}
  [Proof of Theorem \ref{main}]
  Take any simplex $\sigma$ of $K$. Since we are considering a smooth triangulation of $M$, $\mathrm{Int}(\sigma)$ is a submanifold of $M$. Then by the transversality theorem (Theorem~\ref{transversality theorem}), the subset of $C^\infty(L,M)$ consisting of transverse maps to $\mathrm{Int}(\sigma)$ is residual. Thus the intersection of those residual sets, where $\sigma$ ranges over all simplices of $K_{n-p-1}$, is dense since $C^\infty(L,M)$ is a Baire space as mentioned above. In particular, by Lemma \ref{Emb local}, there is an isotopy $h \colon L\times[0,1]\to M$, $(x,t) \mapsto h_t(x)$ such that $h_0$ is the inclusion of $L$ into $M$ and $h_1$ is an embedding of $L$ into $M$ which is transverse to the interior of every simplex of $K_{n-p-1}$. Note that since $\dim L +\dim K_{n-p-1}<\dim M$, we have $h_1(L)\cap K_{n-p-1}=\emptyset$. Now, by the isotopy extension theorem (Theorem~\ref{isotopy extension theorem}), the isotopy $h_t$ extends to a smooth map $H \colon M\times[0,1]\to M$, $(x,t) \mapsto H_t(x)$ such that $H_0=1_M$ and $H_t$ is a diffeomorphism for each $t\in[0,1]$. Thus we get a diffeomorphism
  \[
    H_1\colon M-L\to M-h_1(L)
  \]
  where $K_{n-p-1}\subset M-h_1(L)$. Therefore since $M-L$ is embeddable into $\R^m$, so is $M-h_1(L)$, hence $K_{n-p-1}$, completing the proof.
\end{proof}

Let us consider connected sums. The following higher dimensional version of Schoenflies theorem is proved in \cite{B}.

\begin{lemma}
  \label{Schoenflies theorem}
  Let $f\colon S^{n-1}\times[0,1]\to S^n$ be an embedding. Then there is a homeomorphism $F\colon S^n\to S^n$ such that $F(f(S^{n-1}\times\frac{1}{2}))$ is the equator of $S^n$.
\end{lemma}

\begin{proof}
  [Proof of Proposition \ref{connected sum}]
Clearly, it is sufficient to prove the $k=2$ case. For $i=1,2$, we can take a smooth closed ball $B_i$ in $M_i$ such that $B_i \cap L_i =\emptyset$. Let $\alpha \colon \partial B_1 \to \partial B_2$ be a diffeomorphism. We consider the connected sum $X = (M_1 - L_1) \# (M_2 - L_2)$ by identifying $\partial B_1$ and $\partial B_2$ via $\alpha$.

By assumption, there is a continuous embedding $f_i \colon M_i - L_i \to S^n$. Then, by Lemma~\ref{Schoenflies theorem}, there is a homeomorphism $F_i \colon S^n \to S^n$ such that $F_i (f_i (\partial B_i))$ is the equater $S^{n-1}$ of $S^n$. We may assume that $F_1(f_1(M_1- L_1 - \interior B_1))$ is in the upper hemisphere $S^n_+$ and $F_2(f_2(M_2- L_2 - \interior B_2))$ is in the lower hemisphere $S^n_-$. Let $\phi \colon S^n_- \to S^n_-$ be a homeomorphism that extends the composition of homeomorphisms
\[ S^{n-1} \xrightarrow{(F_2 \circ f_2 |_{\partial B_2})^{-1}} \partial B_2 \xrightarrow{\alpha^{-1}} \partial B_1 \xrightarrow{F_1 \circ f_1} S^{n-1}.\]
Here we regard $S^{n-1}$ as the equator of $S^n$. Then, for $x \in \partial B_1$, we have
\[ \phi \circ F_2 \circ f_2 (\alpha(x)) = F_1 \circ f_1(x).\]
Thus, by gluing $F_1 \circ f_1$ and $\phi \circ F_2 \circ f_2$, we obtain an injective continuous map
\[\Psi \colon X = (M_1 - L_1) \# (M_2 - L_2) \to S^n.\]

Set $A_i = M_i - L_i - \interior B_i$ for $i = 1,2$. Note that $\Psi|_{A_i}$ is an embedding for $i = 1,2$, and that $X = A_1 \cup A_2$. Since $\Psi (A_1) = S^n_+ \cap \Psi(X)$ and $\Psi(A_2) = S^n_- \cap \Psi(X)$ are closed subsets of $\Psi(X)$, it follows that $\Psi$ is an embedding. This completes the proof.
\begin{comment}
Note that $\Psi |_{M_1 - L_1 - \interior B_1}$ and $\Psi |_{M_2 - L_2 - \interior B_2}$ are embedding, and that $M_1 - L_1 - \interior B_1$ and $M_2 - L_2 - \interior B_2$ are a finite closed cover of $(M_1 - L_1) \# (M_2 - L_2)$. These imply that $\Psi$ is an embedding. This completes the proof.
\end{comment}
\end{proof}


\begin{thebibliography}{99}
  \bibitem{B} M. Brown, A proof of the generalized Schoenflies theorem, Bull. Amer. Math. Soc. \textbf{66} (1960), no. 2, 74-76.

  \bibitem{F} A. Flores, \"Uber $n$-dimensionale Komplexe die im $R_{2n+1}$ absolut selbstverschlungen sind, Ergeb. Math. Kolloq. \textbf{6} (1932/1934), 4-7.

  \bibitem{HKTT} S. Hasui, D. Kishimoto, M. Takeda, and M. Tsutaya, Tverberg's theorem for cell complexes, Bull. Lond. Math. Soc. \textbf{55} (2023), no. 4, 1944-1956.

  \bibitem{H} M.W. Hirsch, Differential Topology, Graduate Texts in Mathematics \textbf{33}, Springer-Verlag, 1st edition, New York, 1976.

  \bibitem{vK} E.R. van Kampen, Komplexe in euklidischen R\"aumen, Abh. Math. Seminar Univ. Hamburg \textbf{9} (1933), 72-78.

  \bibitem{KM1} D. Kishimoto and T. Matsushita, Van Kampen-Flores theorem for cell complexes, Discrete Comput. Geom. \textbf{71} (2024), no. 3, 1081-191.

  \bibitem{KM2} D. Kishimoto and T. Matsushita, Van Kampen-Flores theorem and Stiefel-Whitney classes, accepted by Proc. Amer. Math. Soc.

  \bibitem{KrM} A. Kriegl and P.W. Michor, The convenient setting of global analysis, Mathematical Surveys and Monographs \textbf{53} American Mathematical Society, Providence, RI, 1997.

  \bibitem{NW} E. Nevo and U. Wagner, On the embeddability of skeleta of spheres, Israel J. Math. \textbf{174} (2009), 381-402.

  \bibitem{W} J.H.C. Whitehead, On $C^1$-complexes, Ann. of Math. \textbf{41}, no. 4, 809-824.
\end{thebibliography}
\end{document}